\documentclass[12pt,oneside]{amsart}
\usepackage{amssymb,verbatim,enumerate,ifthen}
\usepackage[mathscr]{eucal}
\usepackage[utf8]{inputenc}
\usepackage[T1]{fontenc}
\usepackage{esint}
\usepackage[marginparwidth=2.4cm]{geometry}

\newtheorem{thm}{Theorem}[section]

\newtheorem{prop}[thm]{Proposition}
\newtheorem{lem}[thm]{Lemma}

\theoremstyle{definition}

\newtheorem{rem}[thm]{Remark}

\numberwithin{equation}{section}
\def\eq#1{{\rm(\ref{#1})}}
\def\Eq#1#2{\ifthenelse{\equal{#1}{*}}
  {\begin{equation*}\begin{aligned}[]#2\end{aligned}\end{equation*}}
  {\begin{equation}\begin{aligned}[]\label{#1}#2\end{aligned}\end{equation}}}

\def\E{\mathscr{E}}
\def\G{\mathscr{G}}
\def\M{\mathscr{M}}
\def\P{\mathscr{P}}
\newcommand{\operator}[1]{\mathop{\vphantom{\sum}\mathchoice
{\vcenter{\hbox{\LARGE $#1$}}}
{\vcenter{\hbox{\large $#1$}}}{#1}{#1}}\displaylimits}

\def\Mst_#1^#2{\operator{\mathscr{M}_{\mbox{\scriptsize$\#$}}\!\!}_{#1}^{#2}\,\,}

\newcommand\R{\mathbb{R}}

\newcommand\N{\mathbb{N}}

\newcommand{\abs}[1]{\left| #1 \right| }

\DeclareMathOperator{\sign}{sign}

\newcommand{\Hc}{\mathscr{H}}

\title[Estimating the Hardy constant of nonconcave Gini means]{Estimating the Hardy constant of \\ nonconcave Gini means}

\author{Zsolt P\'ales}
\address{Institute of Mathematics, University of Debrecen, Pf.\ 400, 4002 Debrecen, Hungary}
\email{pales@science.unideb.hu}

\author{Pawe\l{} Pasteczka}
\address{Institute of Mathematics, Pedagogical University of Krak\'ow,  Podchor\k{a}\.{z}ych str 2, 30-084 Krak\'ow, Poland}
\email{pawel.pasteczka@up.krakow.pl}

\thanks{The first author was supported by the K-134191 NKFIH Grant.}



\usepackage{color}

\newcommand{\HQD}{\mathcal{F}} 

\begin{document}
\begin{abstract}
The extension of the Hardy-Knopp-Carleman inequality to several classes of means was the subject of numerous papers.
In the class of Gini means the Hardy property was characterized in 2015 by the second author. The precise value of the associated  Hardy constant was only established for concave Gini means by the authors in 2016. The determination Hardy constant for nonconcave Gini means is still an open problem. The main goal of this paper is to establish sharper upper bounds for the Hardy constant in this case. The method is to construct a homogeneous and concave quasideviation mean which majorizes the nonconcave Gini mean and for which the Hardy constant can be computed.
\end{abstract}
\maketitle

\section{Introduction}

There are a number of recent results, obtained by the authors, which allow one to establish the Hardy constants for various families of means. Most of them are based on the so-called Kedlaya's property \cite{Ked94,Ked99} in the background, which unifies their assumptions. In the most natural setting, we assume that a mean is concave, homogeneous, and repetition invariant (then it is also monotone). These assumptions are relaxed for example using homogenizations techniques \cite{PalPas19a,PalPas20}, or by replacing repetition invariance by a weaker axiom \cite{Pas21a}. 

However we have not been able to relax the concavity assumption till now. Due to this reason, it was difficult to establish an upper estimation for the Hardy constant for Gini means which are nonconcave. We note that some lower estimation can be obtained by using some general results from \cite{PalPas16}. In this paper we apply a certain concavization techniques to improve the trivial upper bound (which can be obtained directly from \cite{PalPas16}) as well as a more recent bound from \cite{Pas2003b}.
For this purpose, we consider Gini means in the quasideviation framework and construct a concave and homogeneous quasideviation mean which is above the given (nonconcave) Gini mean. The problem of calculating the Hardy constant for concave quasideviation means has been solved recently in \cite{PalPas20}. The key auxiliary results are contained in Lemma~\ref{lem:fpq} which enables us the described construction.

\section{Means and their properties}

A function $\M \colon \bigcup_{n=1}^\infty \R_+^n \to \R_+$ such that $\min(x)\le \M(x)\le \max(x)$ holds for all $x$ in the domain of $\M$ is called a \emph{mean (on $\R_+$)}.Throughout the present note all considered means are defined on $\R_+$, thus we can omit the domain of a mean whenever convenient. We also adopt the standard convention that natural properties like convexity, homogeneity, etc. refer to the respective properties of the $n$-variable function $\M|_{\R_+^n}$ to be valid for all $n \in\N$. 

For a given mean $\M$ let $\Hc(\M)$ denote the smallest nonnegative extended real
number, called the \emph{Hardy constant of $\M$}, such that, for all sequences $(x_1,x_2,\dots)$ of positive elements,
\Eq{*}{
\sum_{n=1}^\infty \M(x_1,\dots,x_n) \le \Hc(\M) \cdot \sum_{n=1}^\infty x_n.
}
Means with a finite Hardy constant are called \emph{Hardy means} (cf.\ \cite{PalPas16}).

\subsection{Homogeneous quasideviation means}
Let $\HQD$ be a class of all continuous functions $f\colon\R_+\to\R$ such that 
\begin{enumerate}[(i)]
 \item $\sign(f(t))=\sign(t-1)$ for all $t\in\R_+$, 
 \item for all $x\in (0,1)$, the mapping
$t \mapsto \frac{f(t)}{f(t/x)}$ is strictly increasing on $(x,1)$. 
\end{enumerate}

\begin{lem} \label{L:HQD} Let $f\colon\R_+\to\R$ satisfy (i) of the previous definition and assume that, for some $p\in\R$, the function $f_p(t):=t^pf(t)$ ($t\in\R_+$) is increasing on $\R_+$ and strictly increasing on $(0,1)$. Then $f$ belongs to $\HQD$. In particular, if $f$ is increasing on $\R_+$ and strictly increasing on $(0,1)$, then $f\in\HQD$.
\end{lem}

\begin{proof} We need to verify that condition (ii) is also valid for $f$. Let $x\in(0,1)$ be fixed. Then 
\Eq{*}{
  \frac{f(t)}{f(t/x)}
  =x^p\frac{t^pf(t)}{(t/x)^pf(t/x)}
  =x^p\frac{f_p(t)}{f_p(t/x)}
  =-x^p\frac{(-f_p)(t)}{f_p(t/x)}
}
and, using the monotonicity property of $f_p$, observe that the mappings
\Eq{*}{
 t\mapsto (-f_p)(t) \qquad\mbox{and}\qquad
 t\mapsto \frac{1}{f_p(t/x)}
}
are positive and strictly decreasing/decreasing functions on $(x,1)$. Therefore, their product is strictly decreasing, which implies that $t \mapsto \frac{f(t)}{f(t/x)}$ is strictly increasing on $(x,1)$ and proves that (ii) is valid.
\end{proof}

\begin{rem}
In view of the results os the paper \cite[Theorem 6]{Pal88b}, it turns out that a continuous function $f\colon\R_+\to\R$ belongs to $\HQD$ if and only if, for some $p\in\R$, the function $f_p(t):=t^pf(t)$ ($t\in\R_+$) is increasing and strictly increasing either on $(0,1)$ or on $(1,\infty)$.
\end{rem}

For every $f \in \HQD$ and all positive-entry vector $x:=(x_1,\dots,x_n)$, define $e_{f,x}(y) \colon \R_+ \to \R$ by
\Eq{*}{
e_{f,x}(y):=f\Big(\frac{x_1}y\Big)+\dots+f\Big(\frac{x_n}y\Big).
}
Due to \cite{Pal88a}, it is known that the equation $e_{f,x}(y)=0$ has a unique solution, which is called the \emph{quasideviation mean generated by the quasideviation $E_f(x,y):=f(x/y)$}. We denote it by $y:=\E_f(x_1,\dots,x_n)$. It immediately follows from this definition that $\E_f$ is a homogeneous mean, i.e., $\E_f(tx_1,\dots,tx_n)=t\E_f(x_1,\dots,x_n)$ for all $n\in\N$ and $t,x_1,\dots,x_n>0$. Let us recall two important results for this family of means.

\begin{lem}\label{lem:comphqd}
For all $f,g \in \HQD$ with $f \le g$ we have $\E_f \le \E_g$.
\end{lem}
\begin{proof}
 Fix a vector of positive reals $x:=(x_1,\dots,x_n)$. Then $e_{f,x}(y)\ne0$ for all $y<\E_f(x)$. Furthermore $e_f(\min(x))>0$. Thus, since $e_{f,x}$ is continuous, we get $e_{f,x}(y)>0$ for all $y<\E_f(x)$. 
 Consequently we get $e_{g,x}(y) \ge e_{f,x}(y)>0$ for all $y<\E_f(x)$, which implies $\E_f(x)\le \E_g(x)$.
\end{proof}

\begin{rem} By the results of \cite{DarPal82} and \cite[Theorem 10]{Pal88a}, the comparison inequality $\E_f \le \E_g$ holds if and only if $f\leq ag$ for some positive number $a$.
\end{rem}

\begin{prop}[\cite{PalPas18a}, Theorem~3.4]\label{prop:concqdH}
Let $f\colon\R_+\to\R$ be concave such that $\sign(f(t))=\sign(t-1)$ for all $t\in\R_+$. 
Then $f$ is increasing on $\R_+$ and strictly increasing on $(0,1)$ (and hence $f \in \HQD$), and $\E_f$ is a Hardy mean if and only if 
\Eq{HP}{
\int_0^1 f\Big(\frac1t\Big)\:dt<+\infty.
}
Furthermore, its Hardy constant $c:=\Hc(\E_f)$ is the unique solution of the equation 
\Eq{*}{
\int_0^c f\Big(\frac1t\Big)\:dt=0.
}
\end{prop}

\subsection{Gini means} For a given $p,q \in \R$, define $g_{p,q} \colon \R_+ \to \R$ by 
\Eq{*}{
g_{p,q}(t):=\begin{cases} 
\dfrac{t^p-t^q}{p-q} & \text{ if }p \ne q;\\
t^p\ln t & \text{ if }p = q.
\end{cases}
}
Then, observe that $t\mapsto t^{-q}g_{p,q}(t)$ is strictly increasing. Therefore, by Lemma~\ref{L:HQD}, it follows that $g_{p,q} \in \HQD$ for all $p,q \in \R$. Thus, for $p,q \in \R$, we define the \emph{Gini mean} of parameter $(p,q)$ by $\G_{p,q}:=\E_{g_{p,q}}$ (cf.\ \cite{Gin38}). One can easily see that $\G_{p,q}$ has the following explicit form:
\Eq{GM}{
  \G_{p,q}(x_1,\dots,x_n)
   :=\begin{cases}
    \left(\dfrac{x_1^p+\cdots+x_n^p}
           {x_1^q+\cdots+x_n^q}\right)^{\frac{1}{p-q}} 
      &\mbox{if }p\neq q, \\[4mm]
     \exp\left(\dfrac{x_1^p\ln(x_1)+\cdots+x_n^p\ln(x_n)}
           {x_1^p+\cdots+x_n^p}\right) \quad
      &\mbox{if }p=q.
    \end{cases}.
}
Clearly, in the particular case $q=0$, the mean $\G_{p,q}$ reduces to the $p$th H\"older mean $\P_p$. It is 
also obvious that $\G_{p,q}=\G_{q,p}$. Furthermore by \cite{DarLos70} we know that Gini means are nondecreasing in their parameters, that is $\G_{p,q}\le \G_{p',q'}$ for all $p,q,p',q' \in \R$ with $p\le p'$ and $q \le q'$.

Due to \cite{PalPer04} and \cite{Pas15c} it is known that $\G_{p,q}$ is a Hardy mean if and only if $\min(p,q)\le 0$ and $\max(p,q)<1$. In addition, its Hardy constant is related to the limit 
\Eq{*}{
H_{p,q}:=\lim_{n\to \infty} n \G_{p,q}\big(1,\tfrac12,\dots,\tfrac1n\big)=\begin{cases} 
               \left(\dfrac{1-p}{1-q}\right)^{\frac1{q-p}} & \mbox{if }p \ne q, \\
               e^{\frac1{1-p}} & \mbox{if }p=q=0.
              \end{cases}
}
More precisely, by \cite{PalPas16}, we have that $\Hc(\G_{p,q})=H_{p,q}$ in the case when the Gini mean $\G_{p,q}$ is concave, i.e., when $\min(p,q) \le 0 \le \max(p,q)<1$. In the remaining case, that is, if $\max(p,q)<0$, we only have the lower estimate $H_{p,q}\leq \Hc(\G_{p,q})$. However the exact value of $\Hc(\G_{p,q})$ remains unknown in this quadrant. 

There were two approaches to obtain upper estimations of this value. The comparison criterion easily implies that, for $q<p<0$, we have $\Hc(\G_{p,q}) \le \Hc(\G_{q,0}) = (1-q)^{- \frac1q}$. A sharper upper estimation is due to \cite{Pas2003b}, where it was proved that
\Eq{E:PasBound}{
\Hc(\G_{p,q})\le \begin{cases} \dfrac{\big(\frac{1-q}{1-p}\big)^{\frac{1-p}{p-q}}-p}{1-p} &\text{ for }q<p<0;\\
\frac{e-p}{1-p} &\text{ for }q=p<0.
           \end{cases}
}
The main goal of this paper is to improve both of these upper bounds.

\section{Main results}

Let us first prove a technical result collecting a few properties of the function $g_{p,q}$.

\begin{lem} \label{lem:taupq}
 Let $p,q \in (-\infty,0)$. Then 
 \begin{enumerate}[{\rm(i)}]
  \item $g_{p,q}$ has a global maximum at $\tau_{p,q}\in(1,\infty)$, where
  \Eq{*}{
  \tau_{p,q}:=\begin{cases}\big(\tfrac pq\big)^{\tfrac1{q-p}} & \text{ if }p \ne q, \\[2mm] e^{-1/p}&\text{ if }p=q; \end{cases}
}
  \item $\sign(g_{p,q}(t))=\sign(t-1)$ for all $t\in\R_+$;
\item $g_{p,q}$ restricted to the interval $(0,\tau_{p,q})$ is strictly increasing and concave.
 \end{enumerate}
\end{lem}

\begin{proof}
We prove the assertions only in the case $p > q$. The cases $p=q$ and $p<q$ are fairly analogous. For the brevity, denote $g:=g_{p,q}$ and $\tau:=\tau_{p,q}$. We have
\Eq{*}{
g'(t)=\frac{pt^{p-1}-qt^{q-1}}{p-q} \qquad\text{and}\qquad g''(t)=\frac{p(p-1)t^{p-2}-q(q-1)t^{q-2}}{p-q}.
}
Thus $g'$ has a unique zero which equals $\tau$. The second derivative of $g$ also has a unique zero which equals
$\eta:=\big(\frac{p(p-1)}{q(q-1)}\big)^{\frac1{q-p}}=\tau \big(\frac{p-1}{q-1}\big)^{\frac{1}{q-p}}$. But $q-1<p-1<0$, thus $\frac{p-1}{q-1}>1$, and $q-p>0$ . Therefore $\eta>\tau>1$.

Next, observe that $g(0^+)=-\infty$ and $g(+\infty)=0$. Thus $g$ is strictly increasing on $(0,\tau)$ and strictly decreasing of $(\tau,\infty)$, which implies (i). In particular $g''(\tau)<0$. Consequently $g''$ is negative on the interval $(0,\eta)$. Thus, since $(0,\tau)\subset (0,\eta)$, we obtain (iii).

As we have already describe the monotonicity properties of $g$, to show the property (ii) it is sufficient to verify that $g(1)=0$, which is trivial.
\end{proof}
Now we are ready to prove the crucial lemma, which binds all facts which will be used in the proof of the main theorem.
\begin{lem}\label{lem:fpq}
 Let $p,q \in (-\infty,0)$. Define $f_{p,q} \colon (0,+\infty) \to \R$ by
\Eq{*}{
f_{p,q}(t):=\begin{cases}
          g_{p,q}(t) & \text{ for } t\le \tau_{p,q}; \\
          g_{p,q}(\tau_{p,q}) & \text{ for }t > \tau_{p,q}.
         \end{cases}
}Then 
\begin{enumerate}[{\rm(I)}]
\item $f_{p,q}$ is concave;
\item $\sign(f_{p,q}(t))=\sign(t-1)$ for all $x\in\R_+$;
\item \label{it3} $\Hc(\E_{f_{p,q}})=c_{p,q}$, where $c_{p,q} \in (1,+\infty)$ is the unique  solution of the equation
\Eq{E:solc}{
\frac{g_{p,q}(\tau_{p,q})}{\tau_{p,q}} + \int_{\frac{1}{\tau_{p,q}}}^{c_{p,q}} g_{p,q} \big(\tfrac1t\big)dt=0
}
\item $g_{p,q} \le f_{p,q}$;
\item \label{it5} $\G_{p,q} \le \E_{f_{p,q}}$.
\end{enumerate}
\end{lem}
\begin{proof}
To show (I) we need to verify that $f_{p,q}'$ is decreasing, which easily follows from the definition of $f_{p,q}$, and the Lemma~\ref{lem:taupq}.(iii).
Next, Lemma~\ref{lem:taupq}.(i) yields (IV). Consequently, view of Lemma~\ref{lem:comphqd}, we get (V). Property (II) is implied Lemma~\ref{lem:taupq}.(i) and (ii). 

To proceed to the proof of (III), set $c:=c_{p,q}$, $f:=f_{p,q}$, $g:=g_{p,q}$, and $\tau:=\tau_{p,q}$. It is easy to check that \eq{HP} holds. Thus, by Proposition~\ref{prop:concqdH} we know that $\E_f$ is a Hardy mean and its Hardy constant $c$ satisfies $\tau^{-1}<1<c$ and solves the equation
\Eq{*}{
0&=\int_0^c f(\tfrac{1}{t})\:dt=\int_0^{\frac{1}{\tau}} f(\tfrac{1}{t})\:dt+\int_{\frac{1}{\tau}}^c f(\tfrac{1}{t})\:dt\\
&=\int_0^{\frac1\tau} g(\tau)\:dt+\int_{\frac{1}{\tau}}^c g(\tfrac{1}{t})\:dt
=\frac{g(\tau)}{\tau}+\int_{\frac{1}{\tau}}^c g(\tfrac{1}{t})\:dt,
}
which completes the proof
\end{proof}

Now we are ready to present the main result of this note.

\begin{thm}\label{thm:main}
 If $p,q \in (-\infty,0)$, then $\Hc(\G_{p,q})\le c_{p,q}$, where $c_{p,q}$ is the unique solution $c\in (1,+\infty)$ of the equation 
 \Eq{name}{
\frac{c^{1-q}}{1-q}-\frac{c^{1-p}}{1-p}&=\abs{q}^{\frac{1-p}{q-p}}\abs{p}^{\frac{1-q}{p-q}}\Big(\frac{1}{1-p}-\frac{1}{1-q}\Big) \qquad&\text{ if }p \ne q;\\
c^{1-p}(1-\ln c^{1-p})&=pe^{\frac{1-p}p} \qquad&\text{ if }p=q.
}
\end{thm}
\begin{proof}
Let us denote, as previously, $g:=g_{p,q}$, $\tau:=\tau_{p,q}$, and $c:=c_{p,q}$.
Then applying Lemma~\ref{lem:fpq} \eqref{it5}, and \eqref{it3} we have $\Hc(\G_{p,q})\le\Hc(\E_{f_{p,q}})=c$, where $c$ solves \eq{E:solc}.
 Thus for $p\ne q$ we have
 \Eq{}{
0&=\frac{g(\tau)}{\tau} + \int_{\frac{1}{\tau}}^{c} g \big(\tfrac1t\big)dt
=\frac{g(\tau)}{\tau} +\int_{\frac{1}{\tau}}^{c} \frac{t^{-p}-t^{-q}}{p-q}\:dt\\
&=\frac{g(\tau)}{\tau} + \frac{1}{p-q}\cdot\Big( \frac{c^{1-p}-\tau^{p-1}}{1-p} - \frac{c^{1-q}-\tau^{q-1}}{1-q}\Big)\\
&=\frac{1}{p-q}\cdot\Big(\frac{c^{1-p}}{1-p}-\frac{c^{1-q}}{1-q}+\Big(\frac{(p-q)g(\tau)}{\tau}+\frac{\tau^{p-1}}{p-1} - \frac{\tau^{q-1}}{q-1}\Big)\Big).
 }
 Equivalently, we have
\Eq{1811-1}{
\frac{c^{1-q}}{1-q} -\frac{c^{1-p}}{1-p}=\frac{(p-q)g(\tau)}{\tau}+\frac{\tau^{p-1}}{p-1} - \frac{\tau^{q-1}}{q-1}.
}
In order to complete the proof, observe that
\Eq{*}{
\frac{(p-q)g(\tau)}{\tau}+\frac{\tau^{p-1}}{p-1} - \frac{\tau^{q-1}}{q-1}&=\tau^{p-1}-\tau^{q-1}+\frac{\tau^{p-1}}{p-1} - \frac{\tau^{q-1}}{q-1}=\frac{p\tau^{p-1}}{p-1}-\frac{q\tau^{q-1}}{q-1}\\
&=\frac{p}{p-1} \Big(\frac {\abs{p}}{\abs{q}}\Big)^{\tfrac{p-1}{q-p}}-\frac{q}{q-1} \Big(\frac{\abs{p}}{\abs{q}}\Big)^{\tfrac{q-1}{q-p}}\\
&=\frac{\abs{p}}{1-p} \abs{p}^{\tfrac{p-1}{q-p}}\abs{q}^{\tfrac{1-p}{q-p}}-\frac{\abs{q}}{1-q} \abs{p}^{\tfrac{q-1}{q-p}}\abs{q}^{\tfrac{1-q}{q-p}}\\
&=\abs{q}^{\frac{1-p}{q-p}}\abs{p}^{\frac{1-q}{p-q}}\Big(\frac{1}{1-p}-\frac{1}{1-q}\Big).
}
Therefore \eq{1811-1} implies our assertion in the case $p\ne q$.

Now assume that $p=q<0$. Then \eq{E:solc} simplifies to 
\Eq{*}{
0&=\frac{g(\tau)}{\tau} + \int_{\frac{1}{\tau}}^{c} g \big(\tfrac1t\big)dt
=\frac{g(\tau)}{\tau} -\int_{\frac{1}{\tau}}^{c} t^{-p}\ln t\:dt= \frac{g(\tau)}{\tau} +\Big[ \frac{t^{1-p}(1-\ln t^{1-p})}{(p-1)^2}\Big]_{\frac{1}\tau}^c\\
&=\frac{g(\tau)}{\tau} + \frac{c^{1-p}(1-\ln c^{1-p})}{(p-1)^2}-\frac{\tau^{p-1}(1-\ln \tau^{p-1})}{(p-1)^2}\\
&=\frac{1}{(p-1)^2}\Big(\frac{g(\tau)(p-1)^2}{\tau} + c^{1-p}(1-\ln c^{1-p})-\tau^{p-1}(1-\ln \tau^{p-1})\Big)
}
Thus one gets
\Eq{*}{
 c^{1-p}(1-\ln c^{1-p})=\tau^{p-1}(1-\ln \tau^{p-1})-\frac{g(\tau)(p-1)^2}{\tau}.
}
As in the previous case it is now sufficient to simplify 
\Eq{*}{
\tau^{p-1}(1-\ln \tau^{p-1})-\frac{g(\tau)(p-1)^2}{\tau}
&=\tau^{p-1}(1-\ln \tau^{p-1})-(p-1)\tau^{p-1}\ln\tau^{p-1}\\
&=\tau^{p-1}(1-p\ln\tau^{p-1})
=e^{\frac{1-p}p}(1-p\ln e^{\frac{1-p}p})
=pe^{\frac{1-p}p},
}
which completes the second case.
\end{proof}

We note that equations in \eq{name} are equivalent to \eq{E:solc}.

\section{An example}
In this section we compare the known estimates for the particular Gini mean $\G_{-1,-2}$.

--- First, by $H_{p,q}\le \Hc(\G_{p,q})$, we immediately obtain $1.5\leq\Hc(\G_{-1,-2})$. 

--- In view of the trivial estimation $\G_{-1,-2}\le \G_{0,-2}$, we get $\Hc(\G_{-1,-2}) \le \Hc(\G_{0,-2})=\sqrt{3} \approx 1.732$.

--- Using the inequality \eq{E:PasBound}, we have 
\Eq{*}{
\Hc(\G_{p,q})\le \dfrac{\big(\frac{3}{2}\big)^{2}+1}{2}=\frac{13}{8}=1.625.
}

--- Finally we use Theorem~\ref{thm:main}. Then $c_0$ is a solution of the equation
$\tfrac13c^3-\tfrac12c^2=2^{-2} (\tfrac12-\tfrac13)$ (or, after simplifications, $8c^3-12c^2-1=0$). This polynomial has a unique real root $c_0 \approx 1.552$. Thus $\Hc(\G_{-1,-2}) \le c_0$, which, obviously, improves both previous upper estimations. On the other hand, this is still far from the lower estimation $1.5$.

We now somehow leave apart the Hardy property and examine the mean $\E_{f_{-1,-2}}$, where $f_{-1,-2} \colon \R_+ \to \R$ is defined in Lemma~\ref{lem:fpq} by
\Eq{*}{
 f_{-1,-2}(t)=\begin{cases}
          t^{-1}-t^{-2} & \text{ for } t\le 2; \\
          \tfrac14 & \text{ for }t > 2.
         \end{cases}
}
By Lemma~\ref{lem:fpq}, we already know that $\Hc(\E_{f_{-1,-2}})=c_0$. We are going to show the form of this mean. Since it is symmetric, let $n \in \N$ and $x=(x_1,\dots,x_n)$ be a vector of positive numbers with $x_1 \le \ldots\le x_n$. Put $m:=\E_{f_{-1,-2}}(x)$. Take the maximal $k \in \{1,\dots,n\}$ such that $x_k \le 2m$. Then either $2m < x_{k+1}$ or $k=n$.  Thus
 \Eq{*}{
  0=e_g(m)=\sum_{i=1}^k f_{-1,-2}\Big(\frac{x_i}m\Big)+ \sum_{i=k+1}^n f_{-1,-2}\Big(\frac{x_i}m\Big)= \sum_{i=1}^k \Big(\frac{m}{x_i} -\frac{m^2}{x_i^2}\Big) + \frac{n-k}{4}.
  }
With the notation $s_\alpha:=x_1^\alpha+\dots+x_k^\alpha$, we have $-s_{-2}m^2+s_{-1}m+\tfrac{n-k}{4}=0$.
Therefore, since $m>0$, we get 
\Eq{*}{
m &=\frac{s_{-1} + \sqrt{s_{-1}^2+s_{-2}(n-k)}}{2s_{-2}}= \frac12 \Bigg( \frac{s_{-1}}{s_{-2}}+ \sqrt{\Big(\frac{s_{-1}}{s_{-2}}\Big)^2+\frac{n-k}{s_{-2}}}\Bigg)\\
}
Finally, using the definition of Gini means, we have 
\Eq{*}{
\E_{f_{-1,-2}}(x)&= \frac12 \Big(\G_{-1,-2}(x_1,\dots,x_k)+ \sqrt{\G_{-1,-2}^2(x_1,\dots,x_k)+\tfrac{n-k}k \G_{0,-2}^2(x_1,\dots,x_k)}\Big)\\
&= \frac12 \Big(\G_{-1,-2}+ \sqrt{\G_{-1,-2}^2+\tfrac{n-k}k \G_{0,-2}^2}\ \Big)(x_1,\dots,x_k),
}
where $k\in\{1,\dots,n\}$ is the maximal index satisfying $x_k \le 2\E_{f_{-1,-2}}(x)$.

Remarkably, in the particular case $x_n \le 2\E_{f_{-1,-2}}(x)$ we have $k=n$ and therefore $\E_{f_{-1,-2}}(x)=\G_{-1,-2}(x)$.

\newpage

\begin{thebibliography}{10}

\bibitem{DarLos70}
Z.~Daróczy and L.~Losonczi.
\newblock {Über den {V}ergleich von {M}ittelwerten}.
\newblock {\em Publ. Math. Debrecen}, 17:289–297 (1971), 1970.

\bibitem{DarPal82}
Z.~Daróczy and Zs. Páles.
\newblock {On comparison of mean values}.
\newblock {\em Publ. Math. Debrecen}, 29(1-2):107–115, 1982.

\bibitem{Gin38}
C.~Gini.
\newblock {{D}i una formula compressiva delle medie}.
\newblock {\em Metron}, 13:3–22, 1938.

\bibitem{Ked94}
K.~S. Kedlaya.
\newblock {Proof of a mixed arithmetic-mean, geometric-mean inequality}.
\newblock {\em Amer. Math. Monthly}, 101(4):355–357, 1994.

\bibitem{Ked99}
K.~S. Kedlaya.
\newblock {Notes: {A} {W}eighted {M}ixed-{M}ean {I}nequality}.
\newblock {\em Amer. Math. Monthly}, 106(4):355–358, 1999.

\bibitem{Pas15c}
P.~Pasteczka.
\newblock {On negative results concerning {H}ardy means}.
\newblock {\em Acta Math. Hungar.}, 146(1):98–106, 2015.

\bibitem{Pas2003b}
P.~Pasteczka.
\newblock On a properties of weighted hardy constant for means.
\newblock {\em arXiv.org}, (2003.10953), 2020.

\bibitem{Pas21a}
P.~Pasteczka.
\newblock On the {H}ardy property of mixed means.
\newblock {\em Math. Inequal. Appl.}, 24(3):873--885, 2021.

\bibitem{Pal88a}
Zs. Páles.
\newblock {General inequalities for quasideviation means}.
\newblock {\em Aequationes Math.}, 36(1):32–56, 1988.

\bibitem{Pal88b}
Zs. Páles.
\newblock {Inequalities for differences of powers}.
\newblock {\em J. Math. Anal. Appl.}, 131(1):271–281, 1988.

\bibitem{PalPas16}
Zs. Páles and P.~Pasteczka.
\newblock {Characterization of the {H}ardy property of means and the best
  {H}ardy constants}.
\newblock {\em Math. Inequal. Appl.}, 19(4):1141–1158, 2016.

\bibitem{PalPas18a}
Zs. Páles and P.~Pasteczka.
\newblock {On the best {H}ardy constant for quasi-arithmetic means and
  homogeneous deviation means}.
\newblock {\em Math. Inequal. Appl.}, 21(2):585–599, 2018.

\bibitem{PalPas19a}
Zs. Páles and P.~Pasteczka.
\newblock {On the homogenization of means}.
\newblock {\em Acta Math. Hungar.}, 159(2):537–562, 2019.

\bibitem{PalPas20}
Zs. Páles and P.~Pasteczka.
\newblock {On {H}ardy type inequalities for weighted quasideviation means}.
\newblock {\em Math. Inequal. Appl.}, 23(3):971–990, 2020.

\bibitem{PalPer04}
Zs. Páles and L.-E. Persson.
\newblock {Hardy type inequalities for means}.
\newblock {\em Bull. Austr. Math. Soc.}, 70(3):521–528, 2004.

\end{thebibliography}

\end{document}